\documentclass[a4paper,12pt]{amsart}

\usepackage{a4wide}
\newtheorem{prop}{Proposition}
\newtheorem{lemma}{Lemma}
\newtheorem{rem}{Remark}
\newtheorem{defi}{Definition}
\newtheorem{thm}{Theorem}
\newtheorem{coro}{Corollary}
\usepackage{amssymb}
\usepackage{amsmath}
\usepackage{bm}
\numberwithin{equation}{section}
\numberwithin{prop}{section}
\numberwithin{lemma}{section}
\numberwithin{rem}{section}
\numberwithin{thm}{section}
\numberwithin{defi}{section}
\numberwithin{coro}{section}

\begin{document}

\title{Weyl-Einstein structures on K-contact manifolds}
\author{Paul Gauduchon and Andrei Moroianu}
\address{Paul Gauduchon \\ CMLS\\ {\'E}cole
  Polytechnique \\ CNRS, Universit\'e Paris-Saclay, 91128 Palaiseau \\ France}
\email{paul.gauduchon@polytechnique.edu}

\address{Andrei Moroianu \\ Laboratoire de Math\'ematiques d'Orsay, Universit\'e Paris-Sud, CNRS, Universit\'e Paris-Saclay, 91405 Orsay, France}
\email{andrei.moroianu@math.cnrs.fr}

\begin{abstract}
We show that a compact K-contact manifold $(M,g,\xi)$ has a closed Weyl-Einstein connection compatible with the conformal structure $[g]$ if and only if it is Sasaki-Einstein.
\end{abstract}
\maketitle


\def\d{\mathrm{d}}

\section{Introduction}

{\it $K$-contact} structures  --- see the definition in Section \ref{sKcontact} --- can be viewed as the odd-dimensional counterparts of {\it almost K\"ahler} structure, in the same way as {\it Sasakian} structures are the odd-dimensional counterparts of {\it K\"ahler structures}. It has been shown in \cite{bg}, cf. also \cite{adm}, that {\it compact} Einstein $K$-contact structures are actually Sasakian, hence Sasaki-Einstein. In this note, we consider the more general situation of a compact $K$-contact manifold $(M, g, \xi)$ carrying in addition a {\it Weyl-Einstein} connection $D$ compatible with the conformal class $[g]$,  already considered by a number of authors, in particular in \cite{ghosh} and \cite{ma}. We show --- Theorem \ref{main} and Corollary \ref{cor} below --- that $g$ is then Einstein and  $D$ is the Levi-Civita connection of an Einstein metric $g_0$ in the conformal class $[g]$, which is actually equal $g$ up to scaling, except if $(M, [g])$ is the flat conformal sphere. In all cases, the $K$-contact structure is Sasaki-Einstein.

\section{Conformal Killing vector fields} \label{sconf}

Let $(M, c)$ be a (positive definite) conformal manifold of dimension $n$. A vector field $\xi$ on $M$ is called {\it conformal Killing } with respect to $c$ if it preserves $c$, meaning that for any metric $g$ in $c$, the trace-free part 
$\left(\mathcal{L} _{\xi} g\right) _0$ of the Lie derivative $\mathcal{L} _{\xi} g$ of $g$ along $\xi$ is identically zero, hence that $\mathcal{L} _{\xi} g = f \, g$, for some function $f$, depending on $\xi$ and $g$, and it is then easily checked that $f = - \frac{2\delta ^g \eta _g}{n}$, where $\eta _g$ denotes the $1$-form dual to $\xi$ and $\delta ^g \eta _g$ the co-differential of $\eta_g$ with respect to $g$. In particular, a conformal Killing vector field $\xi$ on $M$ is Killing with respect to some metric $g$ in $c$ if and only if $\delta ^g \eta _g = 0$. In this section, we  present a number of facts  concerning conformal Killing vector fields for further use in this note. 

\begin{prop} \label{l1}
Let $(M,g)$ be a connected compact oriented Riemannian manifold of dimension $n$, $n \geq  2$, carrying a non-trivial parallel vector field $T$. Let $\xi$ be any conformal Killing vector field on $M$ with respect to the conformal class $[g]$ of $g$. Then, $\xi$ is Killing with respect to $g$; moreover,  it commutes 
with $T$ and the inner product $a:= g (\xi, T)$ is constant.
\end{prop}
\begin{proof} Denote by $\eta = \xi ^{\flat}$ the $1$-form dual to $\xi$ and by $\delta \eta$ the co-differential of $\eta$ with respect to $g$;  then, $\xi$ is Killing if and only if $\delta \eta = 0$.  Denote by $\nabla$  the Levi-Civita connection of $g$ and by $\mathcal{L} _T$ the Lie derivative along $T$;   then,  $\nabla _T \xi = [T, \xi] = \mathcal{L} _T \xi$ is conformal Killing, and we have:
\begin{equation} \label{deltanablaTxi} \delta \big(\nabla _T \eta \big) = \delta (\mathcal{L} _T \eta) = \mathcal{L} _T (\delta \eta) = T (\delta \eta). 
\end{equation}
Since $T$ is non-trivial, we may assume  $|T| \equiv 1$. Denote  
 $a = g (\xi, T) = \eta (T)$. Since $\xi$ is conformal Killing, 
$\nabla \xi = A - \frac{\delta \eta}{n} \, {\rm Id}$, where $A$ is skew-symmetric and ${\rm Id}$ denotes the identity; for any vector field $X$ we then have: 
${\rm d} a (X) = g (\nabla _X \xi, T) = - g (\nabla _T \xi, X) - \frac{2 \delta \eta}{n} \, g(X, T)$. We thus get:
\begin{equation}  \label{nablaTxi} \nabla _T \eta = - {\rm d} a - \frac{2 \delta \eta}{n} \, \theta, 
\end{equation}
where $\theta = T ^{\flat}$ denotes the $1$-form dual to $T$. By evaluating both members of (\ref{nablaTxi}) on $T$, we get:
\begin{equation} \label{deltaeta} \delta \eta = - n \, {\rm d} a (T), \end{equation}
whereas,\, by considering their co-differential and by using (\ref{deltanablaTxi}), we get:
\begin{equation} \label{Deltaa} \Delta a = - \frac{(n - 2)}{n} \, T (\delta \eta), \end{equation}
where $\Delta a = \delta {\rm d} a$ denotes the Laplacian of $a$. Denote by $v _g$ the volume form determined by $g$ and the chosen orientation; from (\ref{deltaeta}) and (\ref{Deltaa}), we then infer:
\begin{equation*}\int _M a \, \Delta a \, v _g  =  - \frac{(n - 2)}{n} \, \int _M aT(\delta \eta)\, v _g = \frac{(n - 2)}{n} \, \int _M {\rm d} a (T) \, \delta \eta \, v _g = - \frac{(n - 2)}{n^2} \, \int _M (\delta \eta) ^2 \, v _g, \end{equation*}
hence
\begin{equation} \label{intdadeltaxi} \int _M |{\rm d} a| ^2 \, v _g = \int _M a \, \Delta a \, v _g= - \frac{(n - 2)}{n^2} \, \int _M (\delta \eta) ^2 \, v _g. \end{equation}
This readily implies that ${\rm d} a = 0$ and, either by (\ref{intdadeltaxi}) if $n > 2$ or by (\ref{deltaeta}) if $n = 2$, that $\delta \eta = 0$, i.e. that $\xi$ is Killing. Finally, by (\ref{nablaTxi}) we infer that $\nabla _T \xi = [T, \xi] = 0$. 
\end{proof} 
\begin{rem} {\rm Proposition  \ref{l1} can be viewed as a particular case of a more general statement (Theorem 2.1 in \cite{ms}) concerning conformal Killing forms on Riemannian products.
} \end{rem}

The following well-known Proposition \ref{prop-obata} was first established by T. Nagano in \cite{n} and T. Nagano--K. Yano in \cite{ny} in the more general setting of complete Einstein manifolds. The sketch of proof given here for the convenience of the reader follows M. Obata's treatment in  \cite{obata62}, cf. also \cite{obata} 
for a more general discussion. 
\begin{prop}  \label{prop-obata}
Assume that $(M^n,g)$ is a compact oriented Einstein manifold carrying a conformal Killing vector field which is not Killing. Then $(M,g)$ is, up to constant rescaling, isometric to the round sphere $\mathbb{S}^n$.
\end{prop}
\begin{proof}
We first recall the following lemma, due to A. Lichnerowicz
\cite[\textsection 85]{lichne}, cf. also Theorems 3 and 4 in \cite{obata62}. 
\begin{lemma} \label{lambda} Let $(M, g)$ be a connected compact Einstein manifold of dimension $n \geq 2$ of positive scalar curvature ${\rm Scal}$ $($recall that $\mathrm{Scal}$ is automatically constant for $n\ge 3$ and constant by convention for $n=2)$. Denote by $\lambda _1$ the smallest positive eigenvalue of the Riemannian Laplacian acting on functions. Then, 
\begin{equation} \lambda _1 \geq \frac{{\rm Scal}}{(n - 1)}, \end{equation}
with equality if and only if ${\rm grad} _g f$, the gradient of $f$ with respect to $g$,  is a conformal Killing vector field for each function $f$ in the eigenspace of $\lambda _1$. 
\end{lemma}
\begin{proof} As before denote by $\nabla$ the Levi-Civita connection of the metric $g$ and denote by ${\rm Ric}$ the Ricci tensor of $g$.   For any 1-form $\eta$ on $M$, denote by $\xi:= \eta ^{\sharp}$ the vector field dual to $\eta$ with respect to $g$. The covariant derivative $\nabla \eta$ of $\eta$ then splits as follows:
\begin{equation} \label{dec} \nabla \eta = \frac{1}{2} \left(\mathcal{L} _{\xi} g\right) _0 + \frac{1}{2} {\rm d} \eta - \frac{\delta \eta}{n} \, g, \end{equation}
where $\left(\mathcal{L} _{\xi} g\right) _0$ denotes the trace-free part of $\mathcal{L} _{\xi} g$. By using (\ref{dec}) the {\it Bochner identity} 
\begin{equation} \label{bochner} \Delta \eta = \delta \nabla \eta + {\rm Ric} (\xi) \end{equation}
can be rewritten as
\begin{equation} \label{bochner-univ} {\rm Ric} (\xi) = - \frac{1}{2} \delta \left(\mathcal{L} _{\xi} g\right) _0 + \frac{(n - 1)}{n} \, {\rm d} \delta \eta + \frac{1}{2} \delta {\rm d} \eta. \end{equation}
Let $\lambda$ be any positive eigenvalue of $\Delta$ and $f$ any non-zero element of the corresponding eigenspace, so that $\Delta f = \lambda \, f$. By choosing $\eta := {\rm d} f$, so that $\xi = {\rm grad} _g f$,  and substituting ${\rm Ric} = \frac{\rm Scal}{n} \, g$ in (\ref{bochner-univ}), we get
\begin{equation} \lambda \, {\rm d} f =\Delta {\rm d} f 
=  \frac{\rm Scal}{(n - 1)} \, {\rm d} f + \frac{n}{2(n - 1)} \, \delta \left(\mathcal{L} _\xi g\right)_0. \end{equation}
By contracting  with ${\rm d} f$ and  integrating over $M$, we obtain
\begin{equation} \left(\lambda - \frac{\rm Scal}{(n - 1)}\right) \int _M |{\rm d} f| ^2 \, v _g = \frac{n}{4 (n - 1)} \int _M |\left(\mathcal{L} _\xi  g\right)_0| ^2 \, v _g \geq 0, \end{equation}
so that $\lambda \geq \frac{\rm Scal}{(n - 1)}$, with equality if and only if $\left(\mathcal{L} _\xi g\right)_0 = 0$, hence if and only if  $\xi={\rm grad} _g f$ is conformal Killing. 
\end{proof}
The proof of Proposition \ref{prop-obata} goes as follows. First observe that we may assume ${\rm Scal} > 0$, as any conformal Killing vector field is zero if ${\rm Scal} < 0$ or parallel, hence Killing, if ${\rm Scal} = 0$. Let $\xi$ be any conformal Killing vector field on $M$, with dual $1$-form $\eta$. From (\ref{bochner-univ}), we get:
\begin{equation} \label{crux} {\rm Ric} (\xi) =\frac{\rm Scal}{n} \, \eta = \frac{(n - 1)}{n} \, {\rm d} \delta \eta + \frac{1}{2} \delta {\rm d} \eta, \end{equation}
hence
\begin{equation} \Delta (\delta \eta) = \frac{\rm Scal}{(n - 1)} \, \delta \eta. \end{equation}
From Lemma \ref{lambda}, we then infer  that ${\rm grad} _g (\delta \eta)$ is conformal Killing. By Theorem 5 in \cite{obata}, this implies that $\delta \eta$ is constant, hence identically zero, unless $(M, g)$ is isometric to the standard sphere $\mathbb{S} ^n$. If $(M, g) \neq  \mathbb{S} ^n$,  we then have $\delta \eta = 0$, meaning that $\xi$ is Killing. 
\end{proof}

\section{Weyl-Einstein connections on K-contact manifolds} \label{sKcontact} 

\begin{defi} A K-contact manifold is an oriented  Riemannian manifold $(M,g)$ of odd dimension $n = 2 m + 1$, endowed with a unit Killing vector field $\xi$ whose covariant derivative $\varphi:=\nabla\xi$ satisfies 
\begin{equation}\label{phi2}\varphi^2=-\mathrm{Id}+\eta\otimes\xi,
\end{equation} 
where $\eta$ is the metric dual $1$-form of $\xi$.
\end{defi}

Since $\xi$ is Killing, we have $\d\eta(X,Y)=2g(\varphi(X),Y)$ for all vector fields $X$ and $Y$. The kernel of the 2-form $\d \eta$, equal to that of $\varphi$, is then spanned by $\xi$:
\begin{equation}\label{ker}
\mathrm{ker}(\d\eta)=\mathrm{ker}(\varphi)=\mathbb{R}\xi.
\end{equation}
It follows that the restriction of ${\rm d} \eta$ to $\mathcal{D} := {\rm ker} (\eta)$ is non-degenerate, hence that $\mathcal{D}$ is  a contact distribution  on $M$. Moreover, since $\eta (\xi) = 1$ and $\xi \lrcorner {\rm d} \eta = 0$, $\xi$ is the Reeb vector field of the contact $1$-form $\eta$.

Denote by ${\rm R}$ the Riemannian curvature tensor defined by $\mathrm{R}_{X,Y}:=\nabla_{[X,Y]}-[\nabla_X,\nabla_Y]$. From (\ref{phi2}) we easily infer:
\begin{lemma} \label{Rxi} For any $K$-contact structure, we have:
\begin{equation} \label{R} {\rm R} _{\xi, X} \xi = X - g (\xi, X) \xi, \end{equation}
for any vector field $X$.
\end{lemma}
\begin{proof} We first recall the general {\it Kostant formula}:
\begin{equation} \label{kostant-gen} \nabla _X (\nabla \xi) = {\rm R} _{\xi, X}, \end{equation}
for any vector field $X$ and any Killing vector field $\xi$, on any Riemannian manifold, cf. \cite{kostant}. In the current situation,  we thus have
\begin{equation} \label{kostant} \nabla _X \varphi = {\rm R} _{\xi, X}, \end{equation}
for any vector field $X$. Since $\xi$ is of norm $1$, we infer: ${\rm R} _{\xi, X} \xi = \nabla _X (\nabla _{\xi} \xi) - \nabla _{\nabla _X \xi} \xi = 
- \nabla _{\nabla _X \xi} \xi = - \varphi ^2 (X) = X - g (\xi, X) \, \xi$.
\end{proof}
\begin{rem} \label{rem-sasaki} {\rm A $K$-contact structure $(g, \xi)$ is called a  {\it Sasaki structure} if
  \begin{equation} \label{sasaki} (\nabla _X \varphi) (Y)  = \eta (Y) X - g (X, Y) \xi, \end{equation}
   for any vector fields $X, Y$, or, equivalently in view of (\ref{kostant}), if
   \begin{equation} \label{sasaki-R} {\rm R} _{\xi, X} =
     \xi \wedge X, \end{equation}
  (where the curvature ${\rm R}$ is viewed as a map from $\Lambda ^2 {\rm T}M$ to itself). } \end{rem}

\begin{lemma}[cf. \cite{blair}]\label{ric}
  Viewed as enndomorphism of the tangent bundle via the metric $g$,   the Ricci tensor of any  K-contact manifold satisfies
  \begin{equation} \label{blair} \mathrm{Ric}(\xi) = 2 m  \, \xi. \end{equation}
\end{lemma}
\begin{proof}  From \eqref{kostant} we get:
\begin{equation} \label{nablaxiphi} \nabla _{\xi} \varphi = 0, \end{equation}
and 
\begin{equation} \label{deltaphi} \delta \varphi = {\rm Ric} (\xi) \end{equation}
--- here $\delta \varphi$ denotes the co-differential of the endomorphism $\varphi$ and ${\rm Ric}$ is regarded as a field of endomorphisms of ${\rm T} M$ ---  
whereas, from (\ref{phi2}) we readily infer
\begin{equation} \label{nablaXphi} \nabla _X \varphi \circ \varphi + \varphi \circ \nabla _X \varphi = \frac{1}{2} X \lrcorner {\rm d} \eta \otimes \xi + \eta \otimes \varphi (X), \end{equation}
hence
\begin{equation} \label{nablaXphixi} (\nabla _X \varphi) (\xi) = {\rm R} _{\xi, X} \xi = X - \eta (X) \xi, \end{equation}
for any vector field $X$, from which we get
\begin{equation} \label{ricxixi} {\rm Ric} (\xi, \xi) = n - 1 = 2 m. \end{equation} 
In view of (\ref{ricxixi}) and (\ref{ker}), to prove Lemma
\ref{ric} it is sufficient to check  that $\varphi \big({\rm Ric} (\xi)\big) 
= 0$, or else, by (\ref{deltaphi}), that $\varphi (\delta \varphi) = 0$. In view of  (\ref{nablaxiphi}), we have 
\begin{equation} \delta \varphi = - \sum _{i = 1} ^{2 m} (\nabla _{e _i} \varphi) (e _i), \end{equation}
 for any auxiliary (local) orthonormal frame of $\mathcal{D}$; from (\ref{nablaXphi}) we thus get 
\begin{equation} \varphi (\delta \varphi) = \sum _{i = 1} ^{2 m} (\nabla _{e _i} \varphi) \big(\varphi (e _i)\big). \end{equation}
 Since $\varphi$ is associated to the {\it closed} $2$-form ${\rm d} \eta$, for any vector field $X$ we have:
\begin{equation*} g\big(\sum _{i = 1} ^{2 m} (\nabla _{e _i} \varphi) \Big(\varphi (e _i)\big), X\Big) = - \frac{1}{2} \sum _{i = 1} ^{2 m} g \big((\nabla _X\varphi) (e _i), \varphi (e _i)\big) = - g (\nabla _X \varphi, \varphi), \end{equation*}
 which is equal to zero since the norm of $\varphi$ is constant.
\end{proof}
In the following statement, we denote by $(\mathbb{S} ^{2 m + 1}, c _0)$ the $(2 m + 1)$-dimensional sphere, equipped with the standard   flat conformal structure $c _0$. 
\begin{prop} \label{prop-flat} Let $(g, \xi)$ be any  $K$-contact structure on
 $(\mathbb{S} ^{2 m + 1},c_0)$, such that $g$ belongs to the conformal class $c _0$. Then, $g$ has constant sectional curvature equal to $1$ and the $K$-contact structure is then isomorphic to the standard Sasaki-Einstein structure.
\end{prop}
\begin{proof} Since $c _0$ is flat, the curvature ${\rm R}$ of $g$ is of the form
  \begin{equation} \label{RS} {\rm R} _{X, Y} = {\rm S} (X) \wedge Y + X \wedge {\rm S} (Y), \end{equation}
  where, in general,  for any $n$-dimensional Riemannian manifold $(M, g)$, the {\it normalized Ricci tensor} (or Schouten tensor) ${\rm S}$ is defined by
  \begin{equation} \label{S-gen} {\rm S} = \frac{1}{(n - 2)} \left({\rm Ric}
- \frac{\rm Scal}{2 (n - 1)}{\rm Id}\right). \end{equation}
   It then follows from (\ref{R}), (\ref{blair}),  and (\ref{RS})  that  
 \begin{equation} \label{S} {\rm S} (X) = \frac{1}{(n - 2)} \, \left[\left(\frac{\rm Scal}{2 (n - 1)} - 1\right) \, X + \left(n - \frac{\rm Scal}{(n - 1)}\right) \, g (\xi, X) \, \xi\right] \end{equation}
with $n = 2 m + 1$ (as in (\ref{blair}), in (\ref{S}) and in the sequel of the proof, ${\rm Ric}$ and ${\rm S}$ are regarded as endomorphisms of the tangent bundle via the metric $g$).  In terms of the normalized Ricci tensor ${\rm S}$, the contracted Bianchi identity $\delta {\rm Ric} + \frac{\rm d\, Scal}{2} = 0$, reads
\begin{equation} \label{bianchi} \delta {\rm S} + \frac{\rm d\, Scal}{2 (n - 1)} = 0. \end{equation} By using (\ref{bianchi}), we readily infer from (\ref{S}) that ${\rm Scal}$ is constant, so that
\begin{equation} \label{CY} (\nabla _X {\rm S}) (Y) = \kappa  \, (g (\nabla _X \xi, Y) \xi + g (\xi, Y) \, \nabla _X \xi), \end{equation}
for any vector fields $X, Y$, by setting:
\begin{equation} \label{kappa} \kappa := \frac{1}{(n - 2)} \left(n - \frac{\rm Scal}{(n - 1)}\right). \end{equation}
Since the conformal structure is flat, the
general Bianchi identity (cf. e.g. \cite{der})
\begin{equation} \label{bianchi-gen} \delta {\rm W} _{Z} (X, Y) = (n-3) \, g \big(Z, (\nabla _X {\rm S}) (Y) - (\nabla _Y {\rm S}) (X)\big), \end{equation}
where ${\rm W}$ denotes the Weyl tensor of $g$, implies that $(\nabla _X {\rm S}) (Y)$ is {\it symmetric} in $X, Y$, while, by (\ref{CY}), $g \big((\nabla _X {\rm S}) Y, \xi) = \kappa \, g (\nabla _X \xi, Y)$, which is anti-symmetric, as $\xi$ is Killing;  we thus get $\kappa = 0$, hence by \eqref{kappa}, ${\rm Scal} = n (n - 1)$. By (\ref{S}), this implies ${\rm S} = \frac{1}{2} \mathrm{Id}$, so \eqref{RS} shows that $g$ is a metric of constant sectional curvature equal to $1$.  

Finally, \eqref{kostant} shows that $\nabla_X\varphi=\xi\wedge X$ for every tangent vector $X$, meaning that the K-contact structure is Sasaki-Einstein, and it is well known that the isometry group of $\mathbb{S} ^{2 m + 1}$ acts transitively on the set of Sasaki-Einstein structures on the sphere.
\end{proof} 

\begin{defi}
A Weyl connection on a conformal manifold $(M,c)$ is a torsion-free linear connection $D$ which preserves  the conformal class $c$.
\end{defi}
The latter condition means that for any metric $g$ in the conformal class $c$, there exists a real $1$-form, $\theta ^g$, called the {\it Lee form} of $D$ with respect to $g$, such that $D g = - 2 \theta ^g \otimes g$, and $D$ is then related to the Levi-Civita connection, $\nabla ^g$, of $g$ by
\begin{equation} D _X Y = \nabla ^g _X Y + \theta ^g (X) Y + \theta ^g (Y) X - g (X, Y) \, \left(\theta ^g\right) ^{\sharp _g}, \end{equation}
cf. e.g. \cite{c-p}.
A Weyl connection $D$ is said to be {\it closed} if it is locally the Levi-Civita connection of a (local) metric in $c$, {\it exact} if it is  the Levi-Civita connection of a (globally defined) metric in $c$; equivalently, $D$ is closed, respectively exact,  if its Lee form is closed, respectively exact, with respect to one, hence any, metric in $c$. 

If $M$ is compact, for any Weyl connection on $(M, c)$ there exists a distinguished metric, say $g _0$,  in $c$, usually called the {\it Gauduchon metric} 
of $D$, unique up to scaling, whose Lee form $\theta ^{g _0}$ is co-closed with respect to $g _0$,  \cite{g1}. If $D$ is closed, $\theta ^{g_0}$ is then $g _0$-harmonic, identically zero if $D$ is exact.   

The {\it Ricci tensor}, ${\rm Ric} ^D$,  of a Weyl connection $D$ is the bilinear defined by ${\rm Ric} (X, Y) = {\rm trace} \{Z \mapsto {\rm R} ^D _{X, Z} Y\} = \sum _{i = 1} ^n g ({\rm R} ^D _{X, e _i} Y, e _i)$, for any metric $g$ in $c$ and any $g$-orthonormal basis $\{e _i\} _{i = 1} ^n$.
The Ricci tensor ${\rm Ric} ^D$ defined that way is symmetric if and only if $D$ is closed.

A Weyl connection $D$ is called {\it Weyl-Einstein} if the trace-free component of the symmetric part of ${\rm Ric} ^D$ is identically zero. A closed Weyl-Einstein  connection is locally the Levi-Civita connection of a (local) Einstein metric in $c$; an exact Weyl-Einstein connection is the Levi-Civita connection of a (globally defined) Einstein metric. 

We here recall the following well-known fact, first observed in \cite{tod}, cf. also \cite{g2}. 
\begin{thm} \label{pg} 
Let $D$ be a Weyl-Einstein connection defined  on a compact connected oriented conformal manifold $(M,c)$ and denote by  $g_0$ its Gauduchon metric. Then the vector field $T$ on $M$ dual to the Lee form $\theta^{g_0}$ is Killing with respect to $g_0$. If $D$ is closed,  $T$ is parallel with respect to $g _0$, identically zero if and only if $D$ is exact, and $D$ is then the Levi-Civita connection of $g _0$. 
\end{thm}

The aim of this section is to prove the following:

\begin{thm} \label{main} Let $(M ,g,\xi)$ be a compact K-contact manifold of dimension $n = 2 m + 1$, $m \geq 1$, carrying a closed Weyl-Einstein structure $D$ compatible with the conformal class $c = [g]$. Then $g$ is Einstein and $D$ is the Levi-Civita connection of an Einstein metric $g_0$ in $c$, which is equal to $g$, up to scaling, except if $(M, c)$ is the flat conformal sphere
  $(\mathbb{S}^{2m+1}, c _0)$; in the latter case,  the $K$-contact structure is isomorphic  to the standard Sasaki-Einstein structure of $\mathbb{S} ^{2 m + 1}$.  
\end{thm}

\begin{proof} In view of Proposition \ref{prop-flat}, we may assume that $(M, c)$ is not isomorphic to the flat conformal sphere $(\mathbb{S} ^{2 m + 1}, c _0)$. Let $g_0:=e^{2f}g$ denote the Gauduchon metric of $D$ and let $T$ denote the $g_0$-dual of the Lee form of $D$ with respect to $g_0$. According to Theorem \ref{pg}, $T$ is $\nabla^{g_0}$-parallel. We first show that $T\equiv 0$, i.e. that the closed Weyl-Einstein connection $D$ is actually exact.  
\medskip

Assume, for a contradiction, that $T$ is non-zero. By rescaling the Gauduchon metric $g _0$ if necessary, we may assume that $g_0(T,T)=1$. Denote by $\eta$, resp. $\eta _0$, the $1$-form dual to $\xi$ with respect to $g$, resp. $g _0$. Both $\eta$ and $\eta _0$ are contact $1$-forms for the contact distribution $\mathcal{D}$, and, as already noticed,  $\xi$ is the Reeb vector field of $\eta$. According to Proposition \ref{l1}, $\xi$, which is Killing with respect to $g$, hence conformal Killing with respect to $g_0$, is actually Killing with respect to $g _0$ as well, commutes with $T$,  and the inner product $a := g _0 (\xi, T) = \eta _0 (T)$ is constant;  we then have: $\mathcal{L} _T \eta _0 = 0$, hence that $T \lrcorner {\rm d} \eta _0 = \mathcal{L} _T \eta _0 - {\rm d} \big(\eta _0 (T)\big) = - {\rm d} a = 0$. Moreover, since $\eta _0 (T) = a$ and $T \lrcorner {\rm d} \eta _0=0$,   $a$ cannot be zero --- otherwise, $\eta _0$ would not be a contact $1$-form --- and $\xi _0 := a ^{-1} \, T$ is then the Reeb vector field of $\eta _0$. Since $\eta _0 = e ^{2 f} \, \eta $, the Reeb vector fields $\xi _0$ and $\xi$ are related by
\begin{equation} \label{reebs} \xi _0 = e ^{- 2f} \, \xi + Z _f, \end{equation}
where $Z _f$ is the section of $\mathcal{D}$ defined by
\begin{equation} \label{Zf} (Z _f \lrcorner {\rm d} \eta)_{|\mathcal{D}}= 2 e ^{- 2f} \, {\rm d} f _{|\mathcal{D}}. \end{equation}
Since $M$ is compact, $f$ has critical points and for each  of them, say $x$, 
it follows from (\ref{Zf}) that $Z _f (x) = 0$, hence $\xi _0 (x) = a ^{-1} \, T (x) = e ^{ - 2 f (x)} \, \xi (x)$. Since, $g _0 (T (x), T (x)) = 1$ and $g _0 (\xi (x), T (x)) = a$ for any $x$, we infer that $e ^{ 2 f (x)} = a ^2$ for {\it any} critical point $x$ of $f$, in particular for points where $f$ takes its minimal or its maximal value. It follows that $f$ is {\it constant}, with $e ^{2 f} \equiv a ^2$, that  $g _0 = a ^2 \, g$ and $\xi = a \, T$. In particular, $\eta$ and $\eta _0 = a ^2 \, \eta$ are parallel, with respect to $g$ and $g _0$, hence closed. This contradicts the fact that they are contact $1$-forms. 

\smallskip
In view of the above, $T$ must be  identically zero. This means that $D$ is the Levi-Civita connection of the Gauduchon metric $g_0$, which is thus Einstein. Since $\xi$ with respect to $g$, hence conformal Killing with respect to $g_0=e^{2f}g$, and $(M, c)$ is not isomorphic to the flat conformal sphere $(\mathbb{S} ^{2 m + 1}, c _0)$, it follows from  Proposition \ref{prop-obata} that $\xi$ is Killing with respect to $g_0$ as well. 
We thus have ${\rm d} f (\xi) = 0$, hence 
\begin{equation} \label{perp}g(\xi,{\rm grad} _g f)=0.
\end{equation}
Let $\lambda$ denote the Einstein constant of $(M,g_0)$, so that $\mathrm{Ric}^0=\lambda g_0=e^{2f}\lambda g$.
The classical formula relating the Ricci tensors $\mathrm{Ric}$ and $\mathrm{Ric^0}$ of $g$ and $g_0$ reads (cf. \cite{besse}, p. 59):
\begin{equation}\label{ricci}
\mathrm{Ric}^0=\mathrm{Ric}-(2m-1)(\nabla^g\d f-\d f\otimes\d f)+(\Delta^g f-(2m-1)|\d f|_g^2)g.
\end{equation}

Contracting \eqref{ricci} with $\xi$ and using Proposition \ref{ric} we get
$$\lambda e^{2f}\eta=2m\eta-(2m-1)\nabla^g_\xi\d f+(\Delta^g f-(2m-1)|\d f|_g^2)\eta.$$
Taking the metric duals with respect to $g$ this equation reads
\begin{equation}\label{nabla}\nabla^g_\xi({\rm grad} _g f)=h\xi,\qquad \hbox{with}\qquad\ h:=\frac1{2m-1}\left(\Delta^g f-(2m-1)|\d f|_g^2+2m-\lambda e^{2f}\right).
\end{equation}
On the other hand, we have $0=\d\mathcal{L}_\xi f=\mathcal{L}_\xi \d f$, thus $\mathcal{L}_\xi ({\rm grad} _g f)=0$ and therefore $$\nabla^g_\xi({\rm grad} _g f)=
\nabla^g_{{\rm grad} _g f}\xi=\varphi({\rm grad} _g f).$$ Since the image of $\varphi$ is orthogonal to $\xi$, \eqref{nabla} implies that $\varphi({\rm grad} _g f)=0$, thus 
by \eqref{ker}, ${\rm grad} _g f$ is proportional to $\xi$. From \eqref{perp} we thus get ${\rm grad} _g f=0$, so $f$ is constant and $D$ is the Levi-Civita connection of $g$, and hence $g$ is Einstein.
\end{proof}

As a direct corollary of Theorem \ref{main} above together with Theorem 1.1 in \cite{adm} (see also \cite{bg}), we obtain the following result:

\begin{coro} \label{cor}
If $(M^{2m+1},g,\xi)$ is a compact K-contact manifold carrying a closed Weyl-Einstein structure compatible with $g$, then $M$ is Sasaki-Einstein.
\end{coro}

\begin{rem}
In \cite{ma}  it is claimed  that if $(M^{2m+1},g,\xi)$ is a compact K-contact manifold carrying a compatible closed Weyl-Einstein structure, then $M$ is Sasakian if and only if it is $\eta$-Einstein. Our above result show that the 
hypotheses in \cite{ma} already imply both conditions.
\end{rem}


\begin{thebibliography}{20}



\bibitem{adm} V. Apostolov, T. Draghici, A. Moroianu, {\sl The odd-dimensional Goldberg conjecture}, Math. Nachr. {\bf 279} (2006), 948--952.

\bibitem{besse} A. L. Besse, {\it Einstein manifolds}, Ergebnisse der Mathematik und ihrer Grenzgebiete {\bf  10}, Springer Verlag (1987). 

\bibitem{blair}  D. Blair, {\it Riemannian geometry of contact and symplectic manifolds}, Birkh\"auser, 2002.

\bibitem{bg} C. Boyer, K. Galicki, {\sl Einstein  manifolds  and  contact  geometry,} Proc. Amer. Math. Soc. {\bf 129} (2001), 2419--2430.

\bibitem{c-p} D. M. J. Calderbank, H. Pedersen, {\sl Einstein-Weyl geometry}.  Surveys in differential geometry: essays on Einstein manifolds, Surv. Differ. Geom., VI, Int. Press, Boston, MA (1999), 387--423.

\bibitem{der} A. Derdzi\'nski, {\sl On compact Riemannian manifolds with harmonic curvature}, Math. Ann. {\bf 259} (1982), 145--152

\bibitem{g1} P. Gauduchon, {\sl La 1-forme de torsion d'une vari\'et\'e hermitienne compacte}, Math. Ann. {\bf 267} (1984), 495--518.

\bibitem{g2} P. Gauduchon, {\sl Structures de Weyl-Einstein, espaces de twisteurs et vari\'et\'es de type $S^1\times S^3$}, J. reine angew. Math.
{\bf 469} (1995), 1--50.

\bibitem{ghosh} A. Ghosh, {\it Einstein-Weyl structures on contact metric manifolds}, Ann. Global Anal. Geom. {\bf 35} (2009), 431--441.


  \bibitem{kostant} B. Kostant, {\it Holonomy and the Lie algebra of infinitesimal motions of a Riemann manifold}, Trans. Amer. Math. Soc. {\bf 80} (1955), 528--542.

\bibitem{lichne} A. Lichnerowicz, {\it G\'eom\'etrie des groupes de transformations}, Travaux et recherches math\'ematiques {\bf 3}, Dunod (1958). 

\bibitem{ma} P. Matzeu, {\sl Closed Einstein-Weyl structures on compact Sasakian and cosymplectic manifolds,} Proc. Edinb. Math. Soc. {\bf 54} (2011), 149--160.

\bibitem{ms} A. Moroianu, U. Semmelmann, {\sl Twistor forms on Riemannian products}, J. Geom. Phys. {\bf 58} (2008), 134--1345. 

\bibitem{n} T. Nagano, {\sl The conformal transformation on a space with parallel Ricci tensor,} J. Math. Soc. Japan {\bf 11} (1959), 10--14.

\bibitem {ny} T. Nagano, K. Yano, {\sl Einstein spaces admitting a one-parameter group of conformal transformations,} Ann. of Math. {\bf 69} (1959), 451--461.

\bibitem{obata62} M. Obata, {\sl Certain conditions for a Riemnannian manifold to be isometric with a sphere}, J. Math. Soc. Japan, {\bf 14} (1962), 333--340.

\bibitem{obata} M. Obata, {\sl The conjectures on conformal transformations of Riemannian manifolds}, J. Differ. Geom. {\bf 6} (72) (1971), 247--258. 

\bibitem{tod} K. P. Tod, {\sl Compact $3$-dimensional Einstein-Weyl structures}, J. London Math. Soc. (2) {\bf 45} (1992), 341--351.


\bibitem{uwe}  U. Semmelmann, {\sl Conformal Killing forms on Riemannian manifolds}, Math. Z. {\bf 245} (2003) 503--527.



\end{thebibliography}
\end{document}